\documentclass[11pt,reqno]{amsart}

\usepackage{amsmath,amsthm,amssymb,comment,fullpage}
\usepackage{braket}
\usepackage{mathtools}

\usepackage{caption}
\usepackage{times}
\usepackage[T1]{fontenc}
\usepackage{mathrsfs}
\usepackage{latexsym}
\usepackage[dvips]{graphics}
\usepackage{epsfig}
\usepackage{amsmath,amsfonts,amsthm,amssymb,amscd}
\input amssym.def
\input amssym.tex
\usepackage{color}
\usepackage{hyperref}
\usepackage{url}
\newcommand{\bburl}[1]{\textcolor{blue}{\url{#1}}}

\usepackage{tikz}
\usepackage{tkz-tab}
\usepackage{tkz-graph}
\usetikzlibrary{shapes.geometric,positioning}

\newcommand{\burl}[1]{\textcolor{blue}{\url{#1}}}

\numberwithin{equation}{section}

\newtheorem{thm}{Theorem}[section]

\theoremstyle{plain}

\newtheorem{corollary}[thm]{Corollary}

\newtheorem{lemma}[thm]{Lemma}

\newcommand\be{\begin{equation}}
\newcommand\ee{\end{equation}}
\newcommand\bee{\begin{equation*}}
\newcommand\eee{\end{equation*}}
\newcommand\bea{\begin{eqnarray}}
\newcommand\eea{\end{eqnarray}}
\newcommand\beae{\begin{eqnarray*}}
\newcommand\eeae{\end{eqnarray*}}
\newcommand\bi{\begin{itemize}}
\newcommand\ei{\end{itemize}}
\newcommand\ben{\begin{enumerate}}
\newcommand\een{\end{enumerate}}
\newcommand\bc{\begin{center}}
\newcommand\ec{\end{center}}
\newcommand\ba{\begin{array}}
\newcommand\ea{\end{array}}

\newcommand{\R}{\ensuremath{\mathbb{R}}}

\newcommand{\N}{\mathbb{N}}

\newcommand\frakfamily{\usefont{U}{yfrak}{m}{n}}
\DeclareTextFontCommand{\textfrak}{\frakfamily}

\newcommand{\hr}[1]{\href{#1}{\url{#1}}}

\newcommand{\e}{\varepsilon}

\makeatletter
\g@addto@macro\bfseries{\boldmath}
\makeatother\usepackage{thm-restate}

\title{Extensions of Autocorrelation Inequalities with Applications to Additive Combinatorics}

\author{Sara Fish}
\email{\textcolor{blue}{\href{mailto:sfish@caltech.edu}{sfish@caltech.edu}}}
\address{Department of Mathematics, California Institute of Technology, Pasadena, CA 91126}

\author{Dylan King}
\email{\textcolor{blue}{\href{mailto:kingda16@wfu.edu}{kingda16@wfu.edu}}}
\address{Department of Mathematics, Wake Forest University, Winston-Salem, NC 27109}

\author{Steven J. Miller}
\email{\textcolor{blue}{\href{mailto:sjm1@williams.edu}{sjm1@williams.edu}},  \textcolor{blue}{\href{Steven.Miller.MC.96@aya.yale.edu}{Steven.Miller.MC.96@aya.yale.edu}}}
\address{Department of Mathematics and Statistics, Williams College, Williamstown, MA 01267}

\thanks{This work was supported by NSF grants DMS1659037 and DMS1561945, Wake Forest and Williams College. We thank Charles Devlin VI and Stefan Steinerberger for helpful conversations about this problem.}



\date{\today}

\begin{document}

\maketitle

\begin{abstract}
In a 2019 paper, Barnard and Steinerberger show that for $f\in L^1(\mathbf{R})$, the following autocorrelation inequality holds:
\begin{equation*}
    \min_{0 \leq t \leq 1} \int_\R f(x) f(x+t)\ \mathrm{d}x \ \leq\ 0.411 ||f||_{L^1}^2,
\end{equation*}
where the constant $0.411$ cannot be replaced by $0.37$. In addition to being interesting and important in their own right, inequalities such as these have applications in additive combinatorics where some problems, such as those of minimal difference basis, can be encapsulated by a convolution inequality similar to the above integral. Barnard and Steinerberger suggest that future research may focus on the existence of functions extremizing the above inequality (which is itself related to Brascamp-Lieb type inequalities).

We show that for $f$ to be extremal under the above, we must have
\begin{equation*}
    \max_{x_1 \in \R }\min_{0 \leq t \leq 1} \left[ f(x_1-t)+f(x_1+t) \right] \ \leq\ \min_{x_2 \in \R } \max_{0 \leq t \leq 1} \left[ f(x_2-t)+f(x_2+t) \right] .
\end{equation*}
Our central technique for deriving this result is local perturbation of $f$ to increase the value of the autocorrelation, while leaving $||f||_{L^1}$ unchanged. These perturbation methods can be extended to examine a more general notion of autocorrelation. Let $d,n \in \mathbb{Z}^+$, $f \in L^1$, $A$ be a $d \times n$ matrix with real entries and columns $a_i$ for $1 \leq i \leq n$, and $C$ be a constant. For a broad class of matrices $A$, we prove necessary conditions for $f$ to extremize autocorrelation inequalities of the form
\begin{equation*}
    \min_{ \mathbf{t} \in [0,1]^d } \int_{\R} \prod_{i=1}^n\ f(x+ \mathbf{t} \cdot a_i)\ \mathrm{d}x\ \leq\ C ||f||_{L^1}^n.
\end{equation*}
\end{abstract}

\tableofcontents


\section{Introduction}\label{section:introduction}

\subsection{An autocorrelation inequality}
A recent paper of Bernard and Steinerberger \cite{BS} asks the following: what is the smallest $c$ so that
\begin{align}\label{eqn:constant_c}
\min_{t \in [0,1]} \int_\R f(x) f(x+t)\ \mathrm{d}x \ \leq\ c ||f||_{L^1}^2
\end{align}
holds for any $f \in L^1$.
A first attempt at providing such a $c$ uses Fubini's theorem to show
\begin{align}\label{eqn:bs_fubini}
    \min\limits_{t\in[0,1]}\int_{\R}f(x)f(x+t)\ \mathrm{d}x  &\ \le\  \frac{1}{2}\int_{\R}\int_{\R}f(x)f(x+t)\ \mathrm{d}x\ \mathrm{d}t\nonumber\\
    &\ =\  \frac{1}{2}||f||_{L^1}^{2}.
\end{align}
This shows that $c = 1/2$ satisfies \eqref{eqn:constant_c}, although it is not necessarily the smallest $c$ to do so. This is a crude approximation, since to derive \eqref{eqn:bs_fubini} we replace a minimum with an averaging integral. This bound ($c=1/2$) was improved in \cite{BS}, where the authors show that
\begin{equation}\label{eq:BSresult}
    \min\limits_{t\in[0,1]}\int_{\R}f(x)f(x+t)\ \mathrm{d}x \ \le\   0.411 ||f||_{L^1}^{2}.
\end{equation}
This result was obtained using techniques from Fourier analysis, in particular the Wiener-Khintchine theorem. Barnard and Steinerberger also give explicit examples of functions for which the LHS of \eqref{eq:BSresult} is large, showing that $0.411$ cannot be reduced to $0.37$.

\subsection{A problem in continuous Ramsey theory}
This problem can be viewed as a problem in continuous Ramsey theory.

In discrete mathematics, Ramsey theory is the study of questions of the following form: \emph{How large must a global structure be, in order to guarantee that a smaller substructure appears?} A classic problem in Ramsey theory is the study of Ramsey numbers. Given $r,b \in \N$, let $R(r,b)=:n$ be the smallest natural number such that a 2-coloring of the edges of $K_n$ (with colors red and blue) must contain a red copy of $K_r$ or a blue copy of $K_b$. Ramsey's theorem asserts that $R(r,b)$ exists for all $r, b \in \N$. \textit{Continuous Ramsey theory} is the study of Ramsey-type questions in the continuous setting. An example of such a question is the problem of symmetric subsets, studied by Martin and O'Bryant \cite{MO}. A subset of $[0,1]$ is called \textit{symmetric} if it is invariant under some reflection. Let $\lambda$ denote the $1$-dimensional Lebesgue measure and let
\begin{equation}
    D(x)\ := \ \sup\{r \in \R^+ \text{so that }\forall A \subset[0,1]\text{ with } \lambda(A)=x, \ \exists \ S \subset A \text{ symmetric with } \lambda(S)\geq r\}.
\end{equation}
By placing bounds on $D(x)$, Martin and O'Bryant analyzed the size of symmetric sets $S$ found within larger sets $A$.

Rephrased, \eqref{eqn:constant_c} asks the following. Given a function $f \in L^1$, how does the function
\begin{equation}
    g(t) \ =\  \int_\R f(x) f(x+t)\ \mathrm{d}x
\end{equation}
behave? If we replace $f$ by $c\cdot f$ for some $c \in \R$, then $g(t) =c^2 g(t)$. Thus we must take into account the size of $f$. The $L^1$ norm is a natural choice for measuring the size of a function $f: \R \to \R$. For our measurement of substructure, we follow a recent paper of Bernard and Steinerberger \cite{BS} and investigate, as a function of $t$,
\begin{align}\label{eqn:first_problem}
\frac{\displaystyle\int_{\R}f(x)f(x+t)\ \mathrm{d}x}{||f||_{L^1}^{2}},
\end{align}
where we normalize by $||f||_{L^1}^2$ in order to provide invariance under the scaling discussed above.

The central question behind \eqref{eqn:first_problem} is Ramsey in the following sense. In the classical discrete setting, we are given an edge coloring of some graph of known size, and asked to find which monochromatic subgraphs must appear. Just as there are some colorings which have huge amounts of structure, there are some functions with trivial behavior under $\min_{t \in [0,1]} g(t)$. For example, if ${\rm supp}(f) \subseteq [0,1]$, then $\min_{t \in [0,1]} g(t) = 0$. However, the value
\begin{equation}
    \sup\limits_{f\in L^1}\min_{t \in [0,1]} g(t)
\end{equation}
(corresponding to those graph edge colorings which are least structured) is not well understood.

\subsection{A problem in additive combinatorics}
In addition to being a Ramsey-type problem, this problem is also fundamentally connected to a problem in additive combinatorics.

Given some $n \in \mathbb{N}$, a set of integers $A \subset \mathbb{Z}$ is called a \textit{difference basis} with respect to $n$ if, for the difference set
\begin{equation}
A-A \ :=\ \{a_1-a_2\ |\ a_1,a_2 \in A\},
\end{equation}
we have $\{1,\dots, n \} \subset A-A$. The value
\begin{equation}
    H(n)\ :=\  \min\{|A|,\  A \subset \mathbb{Z} \text{ and } \{1,\dots, n\} \subset A-A \}
\end{equation}
has been studied extensively. The connection to \eqref{eqn:first_problem} is through probability: if $f(n)$ is a probability distribution on $n \in \mathbb{Z}$, then $g(t)$ is the probability distribution given by taking the difference $f-f$. The function $H(n)$ was proposed and studied in \cite{PKHJ,EG,HL,B}. Lower bounds on $H(n)$ as $n \to \infty$ were proved in \cite{L} and later improved in \cite{BT}, while upper bounds were shown in \cite{G}. Since $|A-A|$ is at most quadratic in $|A|$, it is immediate that $H(n) \geq \sqrt{2n}$. In fact, these are the correct asymptotics; the best known results are that
\begin{equation}
    \sqrt{2.435n}\ \le\ H(n) \ \le\ \sqrt{2.645 n}.
\end{equation}
This connection to additive combinatorics motivates our investigation. It is possible that the discrete and continuous problems could inform one another.

\subsection{Our results}
We provide necessary conditions for the existence of a function $f$ maximizing equation \eqref{eqn:first_problem}. This is a question which applies only to continuous Ramsey theory, as opposed to discrete Ramsey theory. In a discrete problem, such as the study of Ramsey numbers, extremal structures trivially exist. For example, given that $R(r,b) = n$, it is clear that there must exist some coloring of $K_{n-1}$ which contains no red copy of $K_r$ or blue copy of $K_b$. Furthermore, since there are but a finite number of such colorings, we know there is only a finite number of such extremal graphs, none `more extreme' than any other. In the continuous case, it is not clear if there exist function(s) maximizing \eqref{eqn:first_problem}.

Our methods are based on perturbation theory. Given a candidate extremal function $f$, we attempt to increase its value under \eqref{eqn:first_problem} by adding a function $g$ which is small in $L^1$ norm. In fact, our perturbation techniques can be extended to prove results on generic convolution-type integrals. If $d, n$ are positive integers, $A$ a $d \times n$ matrix with columns $a_i$ for $1 \leq i \leq n$, and $f \in L^1$, then we study
\begin{equation}\label{eqn:min_t}
    \displaystyle\frac{\min\limits_{\mathbf{t}\in[0,1]^d}\displaystyle\int_{\R} \prod\limits_{i=1}^{n}f(x+\mathbf{t}\cdot a_i)\ \mathrm{d}x}{||f||_{L^1}^{n}}.
\end{equation}



In Section \ref{section:main_results} we state Theorem \ref{thm:min_t}, our main result, as well as Corollary \ref{cor:one_t}, Corollary \ref{cor:many_t}, and Corollary \ref{cor:convex_f}. We prove Theorem \ref{thm:min_t} in Section \ref{section:proofs}, and conclude in Section \ref{section:future_work} with a discussion of potential directions in which our work might be extended.


\section{Main Results}\label{section:main_results}

In Section \ref{subsection:theorem_statement} we state our main result, Theorem \ref{thm:min_t}. The theorem statement relies on a technical result, Lemma \ref{lemma:A_condition}, which we state and prove in Section \ref{subsection:lemma}. In Section \ref{subsection:corollaries} we state Corollary \ref{cor:one_t}, Corollary \ref{cor:many_t}, and Corollary \ref{cor:convex_f}, which are special cases of Theorem \ref{thm:min_t}. Finally, in Section \ref{subsection:discontinuous}, we state conditions under which the continuity hypothesis of Theorem \ref{thm:min_t} can be relaxed.

\subsection{Statement of Theorem \ref{thm:min_t}}\label{subsection:theorem_statement}

We now present our main results on the existence of functions $f$ maximizing \eqref{eqn:min_t}. First we present the theorem in its full generality.

\begin{restatable}{thm}{thetheorem}\label{thm:min_t}
Let $d,n \in \N$ and let $A$ be a $d \times n$ matrix with columns $a_i$ for $1 \leq i \leq n$ satisfying Lemma \ref{lemma:A_condition}. Then a continuous function $f$ maximizing equation \eqref{eqn:min_t} must satisfy both
\begin{equation}
    \max\limits_{x_1 \in \R}\min\limits_{\mathbf{t}\in[0,1]^d}\sum\limits_{i=1}^{n}\prod\limits_{i=1,i\neq j}^{n}f(x_1+\mathbf{t}\cdot (a_i-a_j))\ \le \ \frac{n}{||f||_{L^1}}\min\limits_{\mathbf{t}\in[0,1]^d}\int\limits_{\R}\prod\limits_{i=1}^{n}f(x+\mathbf{t}\cdot a_i)\ \mathrm{d}x
\end{equation}
and
\begin{equation}
    \max\limits_{x_1 \in \R}\min\limits_{\mathbf{t}\in[0,1]^d}\sum\limits_{j=1}^n \prod\limits_{i=1,i\neq j}^{n}f(x_1+\mathbf{t}\cdot (a_i-a_j))\ \le \ \min\limits_{x_2 \in {\rm supp}(f)}\max\limits_{\mathbf{t}\in[0,1]^d}\sum\limits_{j=1}^n \prod\limits_{i=1,i\neq j}^{n}f(x_2+\mathbf{t}\cdot (a_i-a_j)).
\end{equation}
\end{restatable}

This theorem is most easily interpreted in the one-dimensional case, when $d=1$. In this scenario, we find that a function $f$ maximizing equation \eqref{eqn:min_t} must satisfy both
\begin{equation}
    \max\limits_{x_1 \in \R}\min\limits_{t\in[0,1]}\sum\limits_{i=1}^{n}\prod\limits_{i=1,i\neq j}^{n}f(x_1+t\cdot (a_i-a_j))\ \le \ \frac{n}{||f||_{L^1}}\min\limits_{t\in[0,1]}\int\limits_{\R}\prod\limits_{i=1}^{n}f(x+t\cdot a_i)\ \mathrm{d}x
\end{equation}
and
\begin{equation}
    \max\limits_{x_1 \in \R}\min\limits_{t\in[0,1]}\sum\limits_{j=1}^n \prod\limits_{i=1,i\neq j}^{n}f(x_1+t\cdot (a_i-a_j))\ \le \ \min\limits_{x_2 \in {\rm supp}(f)}\max\limits_{t\in[0,1]}\sum\limits_{j=1}^n \prod\limits_{i=1,i\neq j}^{n}f(x_2+t\cdot (a_i-a_j)).
\end{equation}

\subsection{Technical lemmas}\label{subsection:lemma}
Our goal is to study the existence of functions $f$ maximizing \eqref{eqn:min_t}. However, there exist choices of $A$ for which \eqref{eqn:min_t} is unbounded. For example, let $A=\mathbf{0}_{d \times n}$; then equation $\eqref{eqn:min_t}$ is not necessarily even finite for individual $f$. In the specific case $n=d=2$, $\int_\R |f(x)|\ \mathrm{d}x < \infty$ does not imply that $\int_\R f(x)^2\ \mathrm{d}x < \infty$.

Analogous to the reasoning in \eqref{eqn:bs_fubini}, we can use Fubini's Theorem to give a sufficient condition on $A$ for which \eqref{thm:min_t} is bounded from above. These include the choice of $A$ studied in \cite{BS}.

\begin{lemma}\label{lemma:A_condition}
If the $d+1$ by $n$ matrix
\begin{equation*}
    B=\left[\begin{array}{c}
    1 \cdots 1  \\ \hline A
    \end{array}\right]
\end{equation*}
has rank at least $n$, then Equation \eqref{eqn:min_t} is finite for all choices of $f \in L^1$.
\end{lemma}

\begin{proof}
First we see that $d \geq n-1$ is implied by the rank criteria on $A$. Then we observe the right-multiplication
\begin{align}
    \begin{bmatrix} x & t_1 & \dots  &t_d \end{bmatrix}\cdot B &\ =\   \begin{bmatrix}x+\mathbf{t}\cdot a_1& x+\mathbf{t}\cdot a_2& \dots  & x+\mathbf{t}\cdot a_n \end{bmatrix}.
    \intertext{Since $B$ has rank at least $n$, there exists an invertible linear transformation $C: \mathbf{R}^{d} \to \mathbf{R}^{d}$ so that}
    \begin{bmatrix} x & \mathbf{t}\cdot C  \end{bmatrix}\cdot B  & \ = \  \begin{bmatrix} x& x+t_1& \dots  &x+t_{n-1} \end{bmatrix}.
\end{align}

Then we return to our problem and use the sequence of upper bounds
\begin{align}
    \min\limits_{\mathbf{t}\in[0,1]^d}\int\limits_{\R}\prod\limits_{i=1}^{n}f(x+\mathbf{t}\cdot a_i)\ \mathrm{d}x  &\ \le\  \int\limits_{\mathbf{t}\in[0,1]^d}\int\limits_{\R}\prod\limits_{i=1}^{n}f(x+\mathbf{t}\cdot a_i)\ \mathrm{d}x\ \mathrm{d}\mathbf{t}\nonumber\\
    &\ \le\  \int\limits_{\mathbf{t}\in \R^d}\int\limits_{\R}\prod\limits_{i=1}^{n}f(x+\mathbf{t}\cdot a_i)\ \mathrm{d}x\ \mathrm{d}\mathbf{t}.
\end{align}

We exchange the $\mathbf{t}\cdot{} a_i$ for $t_i$ by applying $C$,
\begin{align}
    \int\limits_{\mathbf{t}\in\mathbf{R}^d}\int\limits_{\R}\prod\limits_{i=1}^{n}f(x+\mathbf{t}\cdot a_i)\ \mathrm{d}x\ \mathrm{d}\mathbf{t}  &\ \le\  \int\limits_{\mathbf{t}\in C^{-1}\R^d}\int\limits_{\R}f(x)\prod\limits_{i=1}^{n-1}f(x+t_{i})\ \mathrm{d}x\ \mathrm{d} \mathbf{t}\nonumber\\
&\ \le\  ||f||_{L^1}^{n}.
\end{align}
Thus \eqref{eqn:min_t} is necessarily finite.
\end{proof}

A generalization of Lemma \ref{lemma:A_condition} holds when $d + 1 = n$.

\begin{corollary}\label{cor:A_condition}
Let $d + 1 = n$ and let $A$ be a $n \times n$ matrix with columns $a_i$ such that each $a_i$ contains at least one nonzero entry.
Then Equation \ref{eqn:min_t} is finite for all choices of $f \in L^1$. In fact, we have the bound
\begin{equation}
    \displaystyle\frac{\min\limits_{\mathbf{t}\in[0,1]^d}\displaystyle\int_{\R} \prod\limits_{i=1}^{n}f(x+\mathbf{t}\cdot a_i)\ \mathrm{d}x}{||f||_{L^1}^{n}}\ \leq\ \frac{1}{\sqrt{D}},
\end{equation}
where $D$ is given by
\begin{equation}
    \displaystyle D\ =\ \inf \left\{\frac{\det\left(\sum\limits_{i=1}^{n}\lambda_ia_i'\cdot{}a_i\right)}{\prod\limits_{i=1}^{n}\lambda_i} \ \Bigg| \ \lambda_i \in \mathbb{R}^{>0}  \right\}
\end{equation}
and $a_i'$ denotes the transpose of $a_i$.
\end{corollary}

Corollary \ref{cor:A_condition} follows immediately from the Brascamp-Lieb inequality.

\subsection{Theorem \ref{thm:min_t} for specific $d$, $n$, $A$}\label{subsection:corollaries}

Corollary \ref{cor:one_t} addresses a question asked in \cite{BS} about the existence of $f$ extremizing equation \eqref{eqn:constant_c}. Setting $n = 2$, $d = 1$, and $A = \begin{bmatrix} 0 & 1 \end{bmatrix}$, Theorem \ref{thm:min_t} addresses a question asked in \cite{BS}.

\begin{corollary}\label{cor:one_t}
A continuous function $f$ maximizing
\begin{equation}
    \frac{\min\limits_{\mathbf{t}\in[0,1]}\displaystyle\int_{\R}f(x)f(x+t)\ \mathrm{d}x}{||f||_{L^1}^{2}}
\end{equation}
must satisfy both
\begin{equation}
    \max\limits_{x_1 \in \R}\min\limits_{\mathbf{t}\in[0,1]}[f(x_1-t)+f(x_1+t)]\ \le \ \frac{2}{||f||_{L^1}}\min\limits_{\mathbf{t}\in[0,1]}\int_{\R}f(x)f(x+t)\ \mathrm{d}x
\end{equation}
and
\begin{equation}
    \max\limits_{x_1 \in \R}\min\limits_{\mathbf{t}\in[0,1]}[f(x_1-t)+f(x_1+t)]\ \le \ \min\limits_{x_2 \in \R}\max\limits_{\mathbf{t}\in[0,1]}[f(x_2-t)+f(x_2+t)].
\end{equation}
\end{corollary}

Additionally, setting $d=n$ and $A=I$ in Theorem \ref{thm:min_t}, we find the following.



\begin{corollary}\label{cor:many_t}
Let $n$ be a positive integer, then a continuous function $f$ maximizing
\begin{equation}
    \frac{\min\limits_{\mathbf{t}\in[0,1]^n}\displaystyle\int\limits_{\R}\prod\limits_{i=1}^{n}f(x+t_i)\ \mathrm{d}x}{||f||_{L^1}^{n}}
\end{equation}
must satisfy both
\begin{equation}
    \max\limits_{x_1 \in \R}\min\limits_{\mathbf{t}\in[0,1]^n}\sum\limits_{i=1}^{n}\prod\limits_{i=1,i\neq j}^{n}f(x_1+t_i-t_j)\ \le \ \frac{n}{||f||_{L^1}}\min\limits_{\mathbf{t}\in[0,1]^n}\int\limits_{\R}\prod\limits_{i=1}^{n}f(x+t_i)\ \mathrm{d}x
\end{equation}
and
\begin{equation}
    \max\limits_{x_1 \in \R}\min\limits_{\mathbf{t}\in[0,1]^n}\sum\limits_{j=1}^n \prod\limits_{i=1,i\neq j}^{n}f(x_1+t_i-t_j)\ \le \ \min\limits_{x_2 \in {\rm supp}(f)}\max\limits_{t\in[0,1]^n}\sum\limits_{j=1}^n \prod\limits_{i=1,i\neq j}^{n}f(x_2+t_i-t_j).
\end{equation}
\end{corollary}

\subsection{Theorem \ref{thm:min_t} for convex or concave functions}

The value
\begin{equation}\label{eqn:min_value}
    \max\limits_{x_1 \in \R}\min\limits_{\mathbf{t}\in[0,1]^d}\sum\limits_{j=1}^n \prod\limits_{i=1,i\neq j}^{n}f(x_1+\mathbf{t}\cdot (a_i-a_j))
\end{equation}
found in Theorem \ref{thm:min_t} engenders some discussion on how it is connected to the structure of $f$. Consider the function $r:\R \to \R$ given by
\begin{equation}
    r(x_1)\ = \ \min\limits_{\mathbf{t}\in[0,1]^d}\sum\limits_{j=1}^n \prod\limits_{i=1,i\neq j}^{n}f(x_1+\mathbf{t}\cdot (a_i-a_j)).
\end{equation}
In other words, \eqref{eqn:min_value} asks for $\max_{x}r(x)$. The value $r$ takes on at a given point $x_1$ is not truly global; we may alter $f$ outside of an interval around $x_1$ without altering the value taken there. Nor is it truly local; no amount of information locally around $x_1$ can provide enough information to determine this value, because we allow $t$ to extend to $1$.

In the case $d=1$, equation \eqref{eqn:min_value} reduces to
\begin{equation}
    \max\limits_{x_1 \in \R}\min\limits_{t\in[0,1]}[f(x-t)+f(x+t)].
\end{equation}

When $f$ is convex, the minimum is always obtained for $t=0$, since increasing $t$ increases the symmetric sum about the point $x$. Conversely, for $f$ concave the minimum occurs when $t=1$. This provides the following corollary which follows from simplifying the result of Theorem \ref{thm:min_t} under the additional assumptions that $d = 1$ and $f$ is concave or convex..

\begin{corollary}\label{cor:convex_f}
A continuous convex function $f$ maximizing
\begin{equation}\label{eq:ratio_convex}
    \frac{\min\limits_{\mathbf{t}\in[0,1]}\displaystyle\int_{\R}f(x)f(x+t)\ \mathrm{d}x}{||f||_{L^1}^{2}}
\end{equation}
must satisfy both
\begin{equation}
    \max\limits_{x_1 \in \R}f(x_1)\ \le \ \frac{1}{||f||_{L^1}}\min\limits_{\mathbf{t}\in[0,1]}\int_{\R}f(x)f(x+t)\ \mathrm{d}x
\end{equation}
and
\begin{equation}
    2\max\limits_{x_1 \in \R}f(x_1)\ \le \ \min\limits_{x_2 \in \R}[f(x_2-1)+f(x_2+1)].
\end{equation}
Similarly, a concave function maximizing (\ref{eq:ratio_convex}) must satisfy both
\begin{equation}
    \max\limits_{x_1 \in \R}f(x_1-1)+f(x_1+1)\ \le \ \frac{2}{||f||_{L^1}}\min\limits_{\mathbf{t}\in[0,1]}\int_{\R}f(x)f(x+t)\ \mathrm{d}x
\end{equation}
and
\begin{equation}
    \max\limits_{x_1 \in \R}[f(x_1-1)+f(x_1+1)]\ \le \ 2\min\limits_{x_2 \in \R}f(x_2).
\end{equation}
\end{corollary}

\subsection{Theorem \ref{thm:min_t} for discontinuous functions}\label{subsection:discontinuous}

We may relax the hypothesis on the continuity of $f$ if, in the conditions given in the theorem, we replace $f(x_0)$ with
\begin{equation*}
    \lim_{\varepsilon \to 0} \frac{1}{\varepsilon}\int_{x_0-\varepsilon/2}^{x_0+\varepsilon/2}f(x)\ \mathrm{d}x
\end{equation*}
whenever we evaluate $f$ at a point $x_0$. When $f$ is continuous, this limit is $f(x_0)$. If $f$ has a removable or jump discontinuity (such as those found in a construction in \cite{BS}), the limit is no longer identical to function evaluation, but still may be calculated. By standard results from measure theory, the limit is $f(x_0)$ almost everywhere on $\mathbb{R}$ for any choice of $f$.


\section{Proof of Theorem \ref{thm:min_t}}\label{section:proofs}
Before proving Theorem \ref{thm:min_t}, we recall the theorem statement.
\thetheorem*
\begin{proof}[Proof of Theorem \ref{thm:min_t}]
For $\e > 0$ and $x_1 \in \R$, set $g(x):=\e \chi_{[x_1-\e/2,x_1+\e/2]}$. We show that if the given conditions fail, $f+g$ is an improvement over $f$. That is, we wish to show that
\begin{equation}
    \frac{\min\limits_{\mathbf{t}\in[0,1]^d}\displaystyle\int\limits_{\R}\prod\limits_{i=1}^{n}(f+g)(x+\mathbf{t}\cdot a_i)\ \mathrm{d}x}{||f+g||_{L^1}^{n}}\ \ > \ \  \frac{\min\limits_{\mathbf{t}\in[0,1]^d}\displaystyle\int\limits_{\R}\prod\limits_{i=1}^{n}f(x+\mathbf{t}\cdot a_i)\ \mathrm{d}x}{||f||_{L^1}^{n}}.
\end{equation}
By the triangle inequality it suffices to show that
\begin{equation}
    \frac{\min\limits_{\mathbf{t}\in[0,1]^d}\displaystyle\int\limits_{\R}\prod\limits_{i=1}^{n}(f+g)(x+\mathbf{t}\cdot a_i)\ \mathrm{d}x}{(||f||_{L^1}+||g||_{L^1})^{n}}\ \ > \ \  \frac{\min\limits_{\mathbf{t}\in[0,1]^d}\displaystyle\int\limits_{\R}\prod\limits_{i=1}^{n}f(x+\mathbf{t}\cdot a_i)\ \mathrm{d}x}{||f||_{L^1}^{n}}.
\end{equation}
Now the simple form of $g$ allows us to compute $||g||_{L^1}=\e^2$. By letting $\e \to 0$, we find it is enough to see
\begin{equation}
    \frac{\min\limits_{\mathbf{t}\in[0,1]^d}\displaystyle\int\limits_{\R}\prod\limits_{i=1}^{n}(f+g)(x+\mathbf{t}\cdot a_i)\ \mathrm{d}x}{||f||_{L^1}^n+n\e^2||f||_{L^1}^{n-1}+O(\e^4)}\ \ > \ \  \frac{\min\limits_{\mathbf{t}\in[0,1]^d}\displaystyle\int\limits_{\R}\prod\limits_{i=1}^{n}f(x+\mathbf{t}\cdot a_i)\ \mathrm{d}x}{||f||_{L^1}^{n}}.
\end{equation}
Since $g$ is $O(\e)$, the product on the left hand side will be dominated by those products containing only one $g$. Breaking up the minimum we find the sufficient condition
\begin{equation}
    \frac{\min\limits_{\mathbf{t}\in[0,1]^d}\displaystyle\int\limits_{\R}\sum\limits_{i=1}^{n}g(x+\mathbf{t}\cdot a_j)\prod\limits_{i=1,i\neq j}^{n}f(x+\mathbf{t}\cdot a_i)\ \mathrm{d}x+O(\e^3)}{n\e^2||f||_{L^1}^{n-1}+O(\e^4)}\ \ > \ \  \frac{\min\limits_{\mathbf{t}\in[0,1]^d}\displaystyle\int\limits_{\R}\prod\limits_{i=1}^{n}f(x+\mathbf{t}\cdot a_i)\ \mathrm{d}x}{||f||_{L^1}^{n}}.
\end{equation}
With $f$ continuous, $g(x+\mathbf{t}\cdot a_j)$, as we integrate over $x$, approximates $f(x_1-\mathbf{t}\cdot a_j)$, so that again by letting $\e \to 0$ we need
\begin{equation}
    \frac{\e^2\min\limits_{\mathbf{t}\in[0,1]^d}\displaystyle\sum\limits_{i=1}^{n}\prod\limits_{i=1,i\neq j}^{n}f(x_1+\mathbf{t}\cdot (a_i-a_j))+O(\e^3)}{n\e^2||f||_{L^1}^{n-1}+O(\e^4)}\ \ > \ \  \frac{\min\limits_{\mathbf{t}\in[0,1]^d}\displaystyle\int\limits_{\R}\prod\limits_{i=1}^{n}f(x+\mathbf{t}\cdot a_i)\ \mathrm{d}x}{||f||_{L^1}^{n}}.
\end{equation}
Therefore as higher order terms are eliminated, we find
\begin{equation}
    \min\limits_{\mathbf{t}\in[0,1]^d}\sum\limits_{i=1}^{n}\prod\limits_{i=1,i\neq j}^{n}f(x_1+\mathbf{t}\cdot (a_i-a_j))\ \ > \ \  \frac{n}{||f||_{L^1}}\min\limits_{\mathbf{t}\in[0,1]^d}\int\limits_{\R}\prod\limits_{i=1}^{n}f(x+\mathbf{t}\cdot a_i)\ \mathrm{d}x,
\end{equation}
and taking the maximum over $x_1$ yields the first half of Theorem \ref{thm:min_t}.

For the second half, we take a second point $x_2 \in {\rm supp}(\R)$ and in the spirit of $g$ set $g_1 := \e \chi_{[x_1-\e/2,x_1+\e/2]}$ and $g_2 := \e \chi_{[x_2-\e/2,x_2+\e/2]}$. Then by taking $\e$ small enough we know that $||f+g_1-g_2||_{L^1}=||f||_{L^1}$, and, so long as $x_1 \neq x_2$, $g_1$ and $g_2$ have disjoint support. Then to prove that
\begin{equation}
    \frac{\min\limits_{\mathbf{t}\in[0,1]^d}\displaystyle\int\limits_{\R}\prod\limits_{i=1}^{n}(f+g_1-g_2)(x+\mathbf{t}\cdot a_i)\ \mathrm{d}x}{||f+g_1-g_2||_{L^1}^{n}}\ \ > \ \  \frac{\min\limits_{\mathbf{t}\in[0,1]^d}\displaystyle\int\limits_{\R}\prod\limits_{i=1}^{n}f(x+\mathbf{t}\cdot a_i)\ \mathrm{d}x}{||f||_{L^1}^{n}}
\end{equation}
we show that
\begin{equation}
    \min\limits_{\mathbf{t}\in[0,1]^d}\displaystyle\int\limits_{\R}\prod\limits_{i=1}^{n}(f+g_1-g_2)(x+\mathbf{t}\cdot a_i)\ \mathrm{d}x\ \ > \ \  \min\limits_{\mathbf{t}\in[0,1]^d}\displaystyle\int\limits_{\R}\prod\limits_{i=1}^{n}f(x+\mathbf{t}\cdot a_i)\ \mathrm{d}x.
\end{equation}
By breaking open the minimum and expanding the product on the left hand side, we find
\begin{equation}
    \min\limits_{\mathbf{t}\in[0,1]^d}\int\limits_{\R}\sum\limits_{j=1}^n (g_1-g_2)(x+\mathbf{t}\cdot a_j)\prod\limits_{i=1,i\neq j}^{n}f(x+\mathbf{t}\cdot a_i)\ \mathrm{d}x\ >\ 0.
\end{equation}
Transferring those negative $g_2$ terms to the right we find
\begin{equation}
    \min\limits_{\mathbf{t}\in[0,1]^d}\int\limits_{\R}\sum\limits_{j=1}^n g_1(x+\mathbf{t}\cdot a_j)\prod\limits_{i=1,i\neq j}^{n}f(x+\mathbf{t}\cdot a_i)\ \mathrm{d}x\ \ > \ \  \max\limits_{\mathbf{t}\in[0,1]^d}\int\limits_{\R}\sum\limits_{j=1}^n g_2(x+\mathbf{t}\cdot a_j)\prod\limits_{i=1,i\neq j}^{n}f(x+\mathbf{t}\cdot a_i)\ \mathrm{d}x.
\end{equation}
Once again, $g_1$ and $g_2$, when integrated against a product, return the value of that product evaluated at a specific value of $x$ as $\e \to 0$, giving
\begin{equation}
    \min\limits_{\mathbf{t}\in[0,1]^d}\sum\limits_{j=1}^n \prod\limits_{i=1,i\neq j}^{n}f(x_1+\mathbf{t}\cdot (a_i-a_j))\ \ > \ \  \max\limits_{\mathbf{t}\in[0,1]^d}\sum\limits_{j=1}^n \prod\limits_{i=1,i\neq j}^{n}f(x_2+\mathbf{t}\cdot (a_i-a_j)).
\end{equation}
Taking the best possible $x_1,x_2$ gives the second half of Theorem $\ref{thm:min_t}$.
\end{proof}


\section{Future work}\label{section:future_work}

Theorem \ref{thm:min_t} gives conditions which must hold for any $f$ maximizing \eqref{eqn:min_t}. Future work might be able to show whether or not such functions exist. We have defined a much broader class of convolution-type inequalities than the one studied in \cite{BS}. There, the authors' focus is on placing upper and lower bounds on equation \eqref{eqn:first_problem}. Can we find (the best possible) constants $C_1,C_2$, and function $f_0$, all depending on $d,n,A$, such that for any choice of $f \in L^1$, we have
\begin{equation}
    \frac{\min\limits_{\mathbf{t}\in[0,1]^n}\displaystyle\int\limits_{\R}\prod\limits_{i=1}^{n}f(x+t_i)\ \mathrm{d}x }{||f||_{L^1}^{n}} \ \ \le\  \  C_1
\end{equation}
while
\begin{equation}
    \frac{\min\limits_{\mathbf{t}\in[0,1]^n}\displaystyle\int\limits_{\R}\prod\limits_{i=1}^{n}f_0(x+t_i)\ \mathrm{d}x}{||f_0||_{L^1}^{n}} \ \ \ge \ \  C_2?
\end{equation}

While inspired by autocorrelations found in number theory, connections to the Brascamp-Lieb inequality suggests that further abstraction could prove fruitful. The standard formulation of Brascamp-Lieb considers functions products of functions $f_i$ each operating on domain $\R^{n_i}$, while here we limited ourselves to functions on $\R$. Our perturbation approach appears quite general; could this technique be effective in the most general setting?


\ \\

\end{document}